\documentclass[11pt]{amsart}

\usepackage[T1]{fontenc}
\usepackage{lmodern}
\usepackage{amsmath,amssymb,mathtools,amsthm}
\usepackage{microtype}
\usepackage{booktabs,graphicx}
\usepackage{enumitem}
\usepackage{placeins}
\usepackage{algorithm}
\usepackage{float}
\usepackage[noend]{algpseudocode}
\usepackage[margin=1.05in]{geometry}
\usepackage[hidelinks]{hyperref}

\newtheorem{theorem}{Theorem}[section]
\newtheorem{proposition}[theorem]{Proposition}
\newtheorem{lemma}[theorem]{Lemma}

\theoremstyle{definition}

\newcommand{\C}{\mathbb C}

\newcommand{\interior}{\operatorname{int}}
\newcommand{\disc}{\operatorname{disc}}
\providecommand{\doi}[1]{\href{https://doi.org/#1}{\texttt{doi:#1}}}
\setlist[enumerate]{label=\textup{(\roman*)},leftmargin=*,itemsep=2pt,topsep=4pt}
\allowdisplaybreaks[1]

\title[Simultaneous Newton--Halley Failure in Degree Four]
{Simultaneous Global Convergence Failure of\\
Newton's and Halley's Methods in Degree Four}
\author{Alessio Basti}
\address{Department of Engineering and Geology, ``G. d'Annunzio'' University of Chieti--Pescara, Italy}

\author{Tommaso Cremaschi}
\address{School of Mathematics, Trinity College Dublin, Ireland}

\date{}

\hypersetup{
  pdftitle={Simultaneous Newton--Halley Failure in Degree Four},
  pdfsubject={Simultaneous Newton--Halley Failure in Degree Four}
}

\begin{document}
\begin{abstract}
We construct an explicit quartic polynomial for which Newton's and
Halley's root-finding methods both fail to converge to a root on nonempty open sets of initial values. The example can be found through a reproducible computational search within a one-parameter family of
quartics designed to exhibit a prescribed Newton two-cycle. We then
prove analytically that Newton's method has a superattracting
two-cycle and that Halley's method has a distinct attracting
extraneous two-cycle, using a two-interval contraction argument. This fills a degree-four gap in the simultaneous
Newton--Halley failure problem. The computational approach also identifies
another quartic with numerical evidence of simultaneous failure of Newton's,
Halley's, and Schr\"oder's methods.
\end{abstract}

\maketitle

\section{Introduction and main result}\label{sec:introduction}

Newton's and Halley's methods are among the classical iterations for solving
a nonlinear equation $f(z)=0$. Near a simple root, Newton converges
quadratically and Halley cubically; see, for example,
\cite{Traub,Householder}. Those local orders say little, however, about what
happens when the starting value is not already close to a root. For a
polynomial, the iteration itself is a rational map, and a remote starting
value may be attracted to a periodic orbit that has nothing to do with the
roots \cite{AmatBusquierPlaza}.

For $p\in\C[z]$ we write
\begin{equation}\label{eq:methods}
N_p(z)=z-\frac{p(z)}{p'(z)},\qquad
H_p(z)=z-\frac{2p(z)p'(z)}
{2p'(z)^2-p(z)p''(z)}.
\end{equation}
A root is an attracting fixed point of either map. The global obstruction
considered here is an \emph{extraneous attracting cycle}: an attracting
periodic orbit containing no root of $p$. Such a cycle has an open basin, so
it captures a set of starting values of positive planar measure. In
particular, it prevents \emph{almost-everywhere root convergence}, by which
we mean convergence to a root for planar Lebesgue-almost every initial value
$z_0\in\C$. Global and semilocal results under additional
hypotheses can be found, for instance, in
\cite{DaviesDawson,LingXu}. The complex dynamics of root-finding
iterations, including basin geometry and parameter-space phenomena,
has also been investigated numerically and dynamically; see, for
example, \cite{Basto,Scott,CamposCanelaVindel}.

Already in degree three, neither method is globally reliable on its own.
Roberts and Horgan--Kobelski gave cubic examples with extraneous attracting
cycles for Newton and for Halley separately \cite{RobertsHK}. Collins and
Hunter later constructed simultaneous failures of Newton, Halley, and
Schr\"oder iterations in every degree $d\ge5$ and explicitly noted the
remaining gap in degrees three and four \cite{CollinsHunter}. For cubics,
they conjectured that at least one of Newton's and Halley's methods is
generally convergent. This leaves degree four in a particularly natural
position: can Newton and Halley both fail on the same quartic polynomial?

Our main result gives an affirmative answer.

In particular, we regard \eqref{eq:methods} as rational maps on the Riemann sphere, filling
in removable singularities. A two-cycle is a pair of distinct finite points
$z_1,z_2$ satisfying
\[
R(z_1)=z_2,\qquad R(z_2)=z_1.
\]
When $R$ is holomorphic near the cycle, its multiplier is
\[
\mu=(R^{\circ2})'(z_1)=R'(z_1)R'(z_2),
\]
where $R^{\circ 2}=R\circ R$ denotes the second iterate of $R$. The cycle is attracting if $|\mu|<1$ and superattracting if $\mu=0$.

\begin{theorem}[Quartic obstruction]\label{thm:quartic}
The polynomial
\[
P(z)=24z^4-35z^3-2z+2
\]
has four distinct roots. Its Newton map has the extraneous superattracting
two-cycle $0\leftrightarrow1$. Its Halley map has an extraneous attracting
two-cycle $z_{H,1}\leftrightarrow z_{H,2}$ with
\begin{equation}\label{eq:cycle-location}
z_{H,1}\in\left(-\frac1{10},-\frac{49}{500}\right),\qquad
z_{H,2}\in\left(-\frac{519}{1000},-\frac{259}{500}\right),
\end{equation}
and multiplier satisfying $|\mu_H|<3/10$. Hence neither map has
almost-everywhere root convergence.
\end{theorem}

The example can be recovered systematically: imposing the Newton two-cycle $0\leftrightarrow1$ reduces the quartic
coefficients to a one-parameter family; a deterministic search over rational
values of the parameter, initialized at a distinguished critical point of
the Halley map, produces the value $a=-35/24$ and suggests suitable trapping
intervals. The
claims in Theorem~\ref{thm:quartic} are then proved analytically: the Newton
two-cycle is exact and superattracting, while for Halley's method two
disjoint real intervals are shown to be mapped into one another and the
second iterate is proved to be a contraction.

Thus the degree-four obstruction differs structurally from the
higher-degree construction of Collins and Hunter: the two methods fail for
the same polynomial, but through distinct attracting cycles. In
Section~\ref{sec:discussion} we return to this point and examine whether
Newton's and Halley's methods can share an attracting two-cycle in degree
four. McMullen's rigidity theorem gives a broader degree-four barrier for
generally convergent algebraic root-finding algorithms \cite{McMullen}; the
result above provides a related explicit obstruction for the specific
Newton--Halley pair.

Section~\ref{sec:construction} of the paper describes the numerical search leading to the
quartic. Section~\ref{sec:analytic-proof} gives the analytic proof of
Theorem~\ref{thm:quartic}. Section~\ref{sec:basins} illustrates the dynamics
through basin portraits and representative orbits. Finally,
Section~\ref{sec:discussion} places the quartic result in the context of the
common-cycle problem.

\section{Numerical search for the candidate}\label{sec:construction}

We first analytically construct a family with a superattracting Newton two-cycle, and then numerically search within this family for an attracting Halley two-cycle.

Multiplying a polynomial by a nonzero constant does not change either its Newton or its Halley map, so we may restrict attention to monic quartics. We choose the particularly simple Newton two-cycle $0\leftrightarrow1$ and write
\[
p(z)=z^4+az^3+bz^2+cz+e.
\]
The conditions
\[
N_p(0)=1,\qquad N_p(1)=0,\qquad p''(0)=0
\]
are equivalent to
\[
p(0)=-p'(0),\qquad p(1)=p'(1),\qquad p''(0)=0,
\]
provided $p'(0)p'(1)\ne0$. The last condition makes the Newton cycle superattracting. Solving these equations gives
\[
b=0,\qquad e=-c=3+2a,
\]
and hence the one-parameter family
\begin{equation}\label{eq:family} p_a(z)=z^4+az^3-(3+2a)z+(3+2a),\qquad a\in\C\setminus\{-3/2,-1\}. \end{equation}

It remains to find a parameter for which Halley also has an extraneous
attracting two-cycle. For the search, we restrict to $-3/2<a<-1$, set
$s_a=3+2a>0$, and reflect the variable:
\begin{equation}\label{eq:reflected-family}
F_a(t)=p_a(-t)=t^4-at^3+s_at+s_a.
\end{equation}
All four terms are positive when $t>0$, so $F_a$ has no positive root.
Writing
\[
D_a(t)=2F_a'(t)^2-F_a(t)F_a''(t),\qquad
h_a(t):=H_{F_a}(t)=t-\frac{2F_a(t)F_a'(t)}{D_a(t)},
\]
we have $h_a(t)=-H_{p_a}(-t)$. Thus a positive two-cycle
$u\leftrightarrow v$ of $h_a$ gives an extraneous two-cycle
$-u\leftrightarrow-v$ of $H_{p_a}$.
It therefore suffices to find an attracting positive two-cycle of $h_a$.

Rather than solve the equation $h_a^{\circ2}(t)=t$ directly, we iterate
$h_a$ and look for an orbit that settles into an alternation between two
values. To choose the starting point, we use the fact that every attracting
cycle of a rational map of degree at least two attracts a critical point;
see \cite{Milnor}. This motivates following a critical orbit rather than
starting from an arbitrary point. For our family, the derivative provides
a simple way to choose such a starting point. Differentiating gives
\[
h_a'(t)
=\frac{F_a(t)^2\bigl(3F_a''(t)^2-2F_a'(t)F_a'''(t)\bigr)}{D_a(t)^2}
=\frac{12F_a(t)^2B_a(t)}{D_a(t)^2},
\]
where
\begin{equation}\label{eq:B}
B_a(t)=20t^4-20at^3+6a^2t^2-4s_at+as_a.
\end{equation}
Since $F_a(t)>0$ for $t>0$, a positive zero of $B_a$ is a nonroot
critical point of $h_a$, provided $D_a$ does not vanish there.

The coefficients of $B_a$, in descending powers, have signs
$+,+,+,-,-$. There is only one sign change, so Descartes' rule of signs
gives at most one positive zero. On the other hand,
\[
B_a(0)=as_a<0,\qquad B_a(1)=8a^2-25a+8>0.
\]
Therefore $B_a$ has exactly one positive zero,
which we denote by $c_a\in(0,1)$. We approximate it by bisection on
$[0,1]$ and use it as the starting point. This supplies a consistent choice
for every parameter, but does not imply that this particular critical orbit
detects every attracting cycle.

We test rational parameters $a=-b/q$ with
\[
2\le q\le64,\qquad q<b<\frac{3q}{2},\qquad \gcd(b,q)=1,
\]
in increasing order of $q$, and then of $b$. These inequalities place $a$
in the chosen interval, while the coprimality condition avoids testing the
same rational number more than once. For each parameter, we first compute
$200$ iterates, allowing time for the initial transient to decay, and then
set $u=t_{200}$ and $v=h_a(u)$. A two-cycle would satisfy $h_a(v)=u$, so we
compute the closure residual and the estimated multiplier modulus
\[
\varepsilon=|h_a(v)-u|,\qquad m=|h_a'(u)h_a'(v)|.
\]
The parameter is accepted as a numerical candidate only when
\begin{equation}\label{eq:screen}
u>0,\qquad v>0,\qquad |u-v|>10^{-3},\qquad
\varepsilon<10^{-10},\qquad m<\frac12.
\end{equation}
Positivity places the candidate in the root-free region. The separation
test distinguishes two different values from an approximate fixed point,
and the small closure residual tests whether the orbit nearly repeats
after two steps. Finally, the multiplier test requires the estimated
modulus to lie well below the attraction threshold $1$.

\begin{algorithm}[htbp]
\caption{Search for an extraneous attracting Halley two-cycle}
\label{alg:search}
\begin{algorithmic}[1] 
\For{$q=2,\ldots,64$}
  \For{increasing integers $b$ with $q<b<3q/2$ and $\gcd(b,q)=1$}
    \State Set $a=-b/q$.
    \State Approximate the zero $c_a$ of $B_a$ by $60$ bisection steps on
           $[0,1]$; call the result $t$.
    \State Apply $t\gets h_a(t)$ a total of $200$ times.
    \State $u\gets t$, $v\gets h_a(u)$.
    \State $\varepsilon\gets|h_a(v)-u|$,
           $m\gets|h_a'(u)h_a'(v)|$.
    \If{the conditions in \eqref{eq:screen} hold}
      \State \Return $(a,u,v)$ as a numerical candidate.
    \EndIf
  \EndFor
\EndFor
\State \Return no candidate found.
\end{algorithmic}
\end{algorithm}

The first accepted parameter is the $89$th tested one:
\[
a=-\frac{35}{24},\qquad
P=24p_{-35/24}=24z^4-35z^3-2z+2.
\]
Returning to the original variable gives the numerical estimates
\begin{equation}\label{eq:numerical-cycle}
z_{H,1}\approx-0.098810,\qquad
z_{H,2}\approx-0.518476,
\end{equation}
with estimated multiplier of about $0.0072$. For the reflected orbit, these
values suggest the rational intervals
\begin{equation}\label{eq:intervals}
I=\left[\frac{49}{500},\frac1{10}\right],\qquad
J=\left[\frac{259}{500},\frac{519}{1000}\right].
\end{equation}
The next section independently proves, for $h:=h_{-35/24}$, that
\[
h(I)\subset\interior J,\qquad h(J)\subset\interior I,
\]
and that $h^{\circ2}:I\to I$ is a contraction. Its fixed point $u_*\in I$
and its image $v_*=h(u_*)\in J$ then satisfy
$h(u_*)=v_*$ and $h(v_*)=u_*$. Since $I$ and $J$ are disjoint and positive,
this is an extraneous two-cycle. The contraction estimate gives
$|h'(u_*)h'(v_*)|<1$, so the cycle is attracting. 

\section{Proof of Theorem~\ref{thm:quartic}}\label{sec:analytic-proof}

We prove the existence of the Halley two-cycle by finding two disjoint
intervals that are mapped into one another, with a contractive return
map on either interval. We begin with an elementary criterion for this
two-step mechanism. The proofs of all lemmas and propositions are postponed to Appendix \ref{app:ausiliary}.

\begin{lemma}[Two-interval contraction]\label{lem:contraction}
Let $I$ and $J$ be disjoint nonempty closed intervals, and let $h$ be
continuously differentiable near $I\cup J$. If
\[
h(I)\subset J,\qquad h(J)\subset I,
\]
and
\[
\left(\sup_I |h'|\right)\left(\sup_J |h'|\right)<1,
\]
then $h$ has a unique two-cycle $t_1\leftrightarrow t_2$ in $I\cup J$,
and its multiplier satisfies
\[
|h'(t_1)h'(t_2)|<1.
\]
In particular, if $h$ extends holomorphically to a neighborhood of the
cycle, then the cycle is attracting as a complex dynamical cycle.
\end{lemma}

We apply the lemma to $h=h_{-35/24}$ and the intervals chosen
in~\eqref{eq:intervals}. To verify the two inclusions, we locate the
image of each center and then control how much $h$ varies across the
corresponding interval. More precisely, writing
\[
I=[c_I-r_I,c_I+r_I],\qquad J=[c_J-r_J,c_J+r_J],
\]
a bound $|h'|\le L_I$ on $I$ gives
\begin{equation}\label{eq:center-control}
|h(x)-c_J|
\le |h(c_I)-c_J|+L_Ir_I,
\qquad x\in I.
\end{equation}
Thus $h(I)\subset\interior J$ whenever the right-hand side is smaller
than $r_J$. The other inclusion follows from the analogous estimate
on $J$. We first establish the derivative bounds needed for these
comparisons.

By the scaling invariance noted in Section~\ref{sec:construction},
we may replace $F_{-35/24}$ by its integer multiple
\[
F(t):=24F_{-35/24}(t)=P(-t)
=24t^4+35t^3+2t+2,
\qquad h=H_F.
\]
The derivative formula from Section~\ref{sec:construction}, specialized
to this polynomial, reads
\begin{equation}\label{eq:hprime}
h'(t)=\frac{24F(t)^2K(t)}{D(t)^2},
\end{equation}
where
\begin{equation}\label{eq:DK}
\begin{aligned}
D(t)&=2F'(t)^2-F(t)F''(t),\\
K(t)&=5760t^4+8400t^3+3675t^2-96t-35.
\end{aligned}
\end{equation}
It therefore suffices to bound the numerator factors and keep $D$
away from zero.

\begin{lemma}[Bounds on the proposed intervals]\label{lem:bounds}
The polynomials $F$, $K$, and $D$ satisfy the bounds in
Table~\ref{tab:bounds}. In particular, $h$ has no poles on $I\cup J$,
and
\begin{equation}\label{eq:derivative-bounds}
\sup_I|h'|<\frac15,\qquad
\sup_J|h'|<\frac32.
\end{equation}
\end{lemma}

\begin{table}[htbp]
\centering
\caption{Polynomial bounds on the intervals in~\eqref{eq:intervals}.}
\label{tab:bounds}
\begin{tabular}{c c c c}
\hline
Interval & $F(t)$ & $K(t)$ & $D(t)$\\
\hline
$I$ & $0<F<9/4$ & $|K|<6/5$ & $D<-32$\\
$J$ & $0<F<39/4$ & $0<K<2500$ & $D>1990$\\
\hline
\end{tabular}
\end{table}
We can now prove the result.
\begin{proof}[Proof of Theorem~\ref{thm:quartic}]
The Newton cycle, including its superattracting character, was already
established in Section~\ref{sec:construction}. We therefore only have
to complete the argument for Halley's method.

The centers and radii of the intervals in~\eqref{eq:intervals} are
\[
c_I=\frac{99}{1000},\qquad r_I=\frac1{1000},\qquad
c_J=\frac{1037}{2000},\qquad r_J=\frac1{2000}.
\]
Exact rational evaluation gives
\begin{equation}\label{eq:centers}
|h(c_I)-c_J|<\frac1{40000},\qquad
|h(c_J)-c_I|<\frac1{5000}.
\end{equation}
Combining these estimates with Lemma~\ref{lem:bounds}, for $x\in I$
we obtain
\[
|h(x)-c_J|
\le |h(c_I)-c_J|+\frac15|x-c_I|
<
\frac1{40000}+\frac15r_I
=\frac9{40000}<r_J.
\]
Thus $h(I)\subset\interior J$. Similarly, for $y\in J$,
\[
|h(y)-c_I|
\le |h(c_J)-c_I|+\frac32|y-c_J|
<
\frac1{5000}+\frac32r_J
=\frac{19}{20000}<r_I,
\]
and hence $h(J)\subset\interior I$.

Moreover,
\[
\left(\sup_I|h'|\right)\left(\sup_J|h'|\right)
<\frac15\cdot\frac32=\frac3{10}<1.
\]
Lemma~\ref{lem:contraction} therefore gives a unique two-cycle
$t_1\leftrightarrow t_2$ alternating between the two intervals, with
\[
t_1\in\interior I,\qquad t_2\in\interior J,
\]
and multiplier of modulus strictly smaller than $3/10$.

By the reflection argument of Section~\ref{sec:construction},
\[
z_{H,1}=-t_1,\qquad z_{H,2}=-t_2
\]
form a two-cycle of $H_P$. Since $F>0$ on $I\cup J$, neither point is
a root of $P$, so the cycle is extraneous. The bounds on $D$ in
Lemma~\ref{lem:bounds} show that $h$, and hence $H_P$, is holomorphic
near the cycle. It is therefore an attracting complex cycle with a
nonempty open basin.

Finally,
\[
\disc(P)=-167753740\ne0,
\]
so $P$ has four distinct roots. Thus both Newton's and Halley's methods
possess extraneous attracting two-cycles, and neither has
almost-everywhere root convergence.
\end{proof}

\section{Numerical simulations}\label{sec:basins}
The analytic proof establishes the two cycles and their open complex basins. The simulations below serve a complementary purpose, i.e. showing where root-attracted and cycle-attracted behavior is observed, and how an iteration can approach a nonroot cycle while its polynomial residual remains nonzero. 

For both methods, we sample the same two regions of the complex plane,
\begin{equation}\label{eq:windows}
\begin{aligned}
\Omega_{\mathrm{wide}}&=[-0.8,1.8]+i[-0.975,0.975],\\
\Omega_{\mathrm{local}}&=[-0.12,0.02]+i[-0.0525,0.0525].
\end{aligned}
\end{equation}
Each portrait uses a uniform $1600\times1200$ grid of cell centers, a budget of $200$ iterations, and absolute proximity tolerance $10^{-8}$. The two methods use identical windows, equal-aspect axes, and the same color key. 

The four labels $R_1,\ldots,R_4$ record proximity to the corresponding root references $r_1,\ldots,r_4$. The vermilion class $C$ records two successive iterates close to opposite points of the method's nonroot two-cycle. Light gray, labeled $U$, is reserved for an orbit that remains unclassified. 

Figure~\ref{fig:overview} gives the common overview. All four roots and both points of each nonroot cycle are marked; the dashed rectangle locates the smaller window. Most sampled initial values at this scale are assigned to roots. The small cycle-labeled regions are easier to see in Figure~\ref{fig:detail}, which shows exactly the same local window for Newton and Halley. It contains $0$ for Newton and $z_{H,1}$ for Halley, but no polynomial root.

\begin{figure}[htbp]
\centering
\includegraphics[width=.485\linewidth]{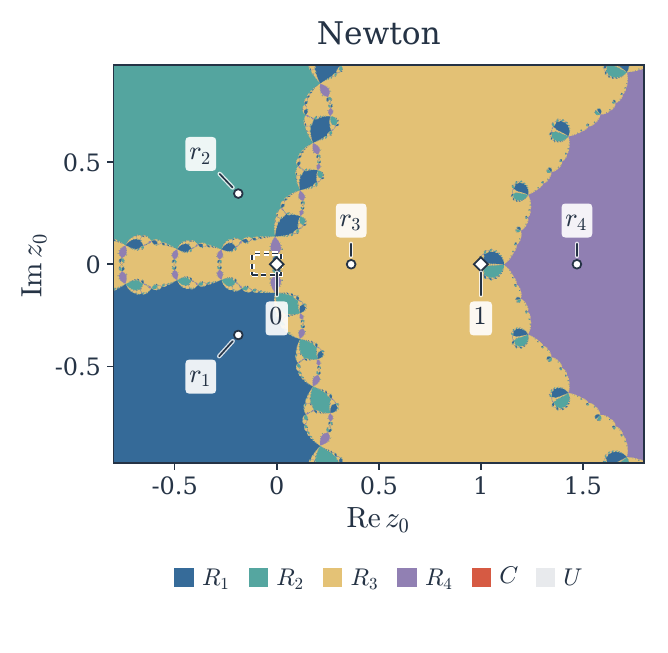}\hfill
\includegraphics[width=.485\linewidth]{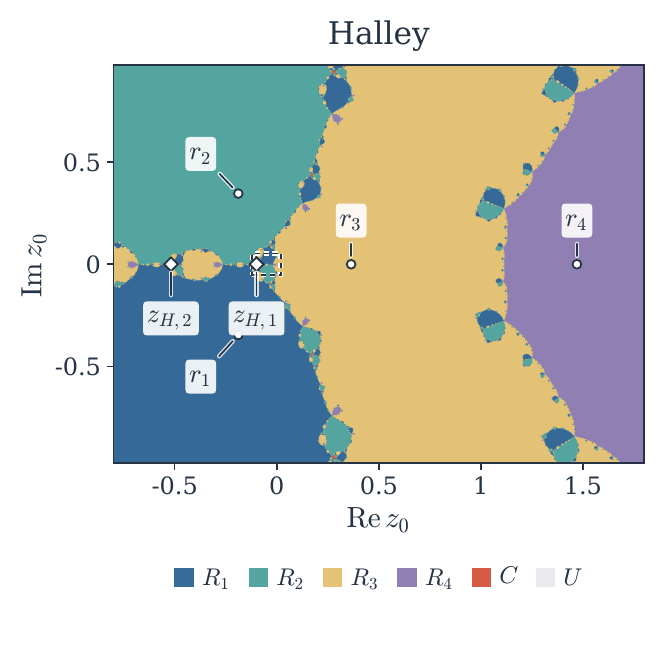}
\caption{Numerical basin portraits in $\Omega_{\mathrm{wide}}$. Circles mark roots and diamonds mark the method's nonroot two-cycle. The dashed rectangle identifies the common zoom in Figure~\ref{fig:detail}.}
\label{fig:overview}
\end{figure}

\begin{figure}[htbp]
\centering
\includegraphics[width=.485\linewidth]{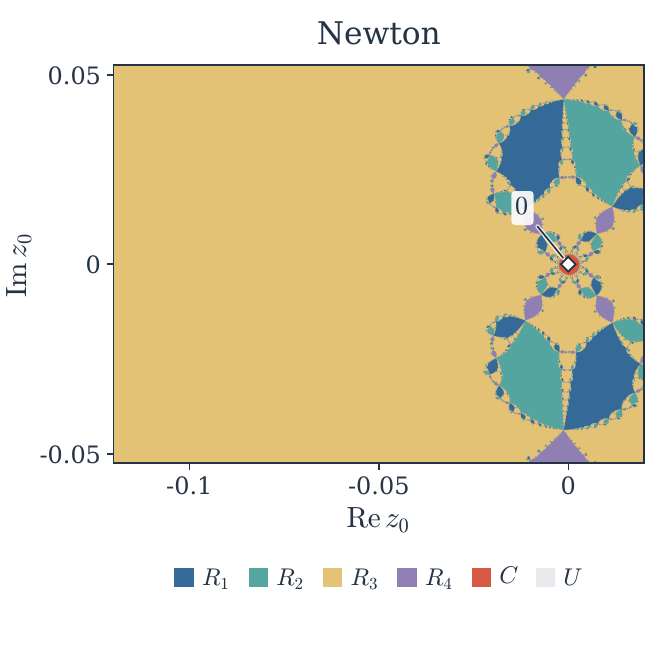}\hfill
\includegraphics[width=.485\linewidth]{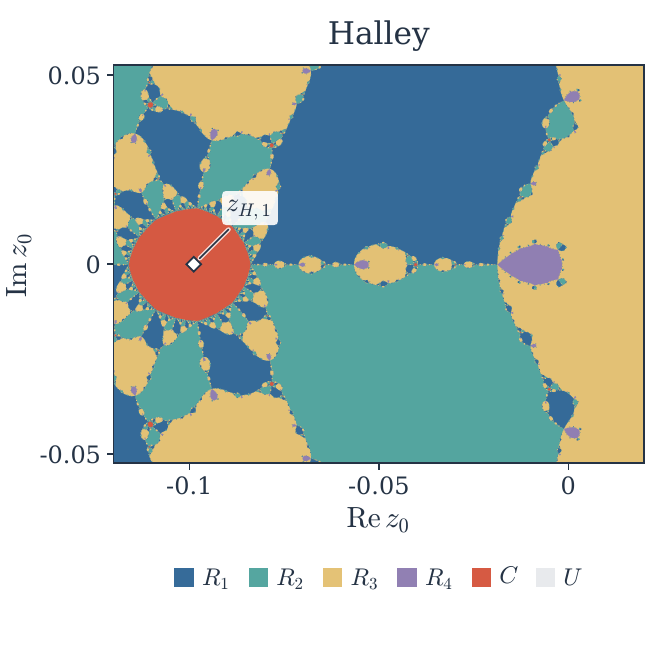}
\caption{Detail in the common window $\Omega_{\mathrm{local}}$, with unchanged colors. The vermilion patches are assigned numerically to $0\leftrightarrow1$ for Newton and to $z_{H,1}\leftrightarrow z_{H,2}$ for Halley. Only one point of each cycle lies in the window.}
\label{fig:detail}
\end{figure}

The larger cycle-colored patch for Halley in this zoom is a feature of the selected window, not a comparison of global reliability. Since $P$ has real coefficients, complex conjugation exchanges the root labels $R_1$ and $R_2$ and preserves the other classes, as reflected in the portraits.

To complement the spatial portraits, consider the real starting values $z_0=10^{-3}$ for Newton and $z_0=-0.099$ for Halley. For either iteration $z_{n+1}=R_P(z_n)$, define
\begin{equation}\label{eq:diagnostics}
\delta_k=|z_{2k+2}-z_{2k}|,\qquad
\rho_k=|P(z_{2k})|,\qquad k=0,\ldots,6.
\end{equation}
Figure~\ref{fig:orbits} plots these quantities from computations performed at $100$ decimal digits. The two-step difference decreases rapidly in both examples, while the polynomial residual stays away from zero. For Newton, the observed even iterates lie increasingly close to $0$, where $P(0)=2$. For Halley, the corresponding residual is numerically close to $2.2337$. A small two-step difference therefore records near-recurrence, not progress toward a root. 
\begin{figure}[htbp]
\centering
\includegraphics[width=.485\linewidth]{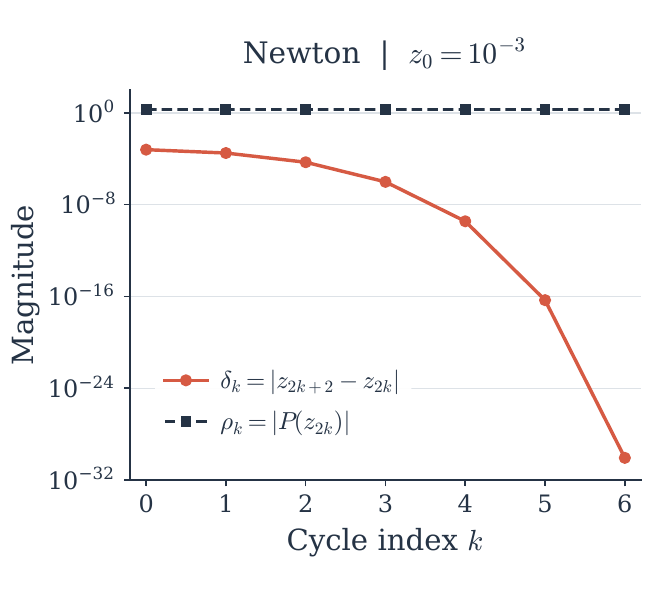}\hfill
\includegraphics[width=.485\linewidth]{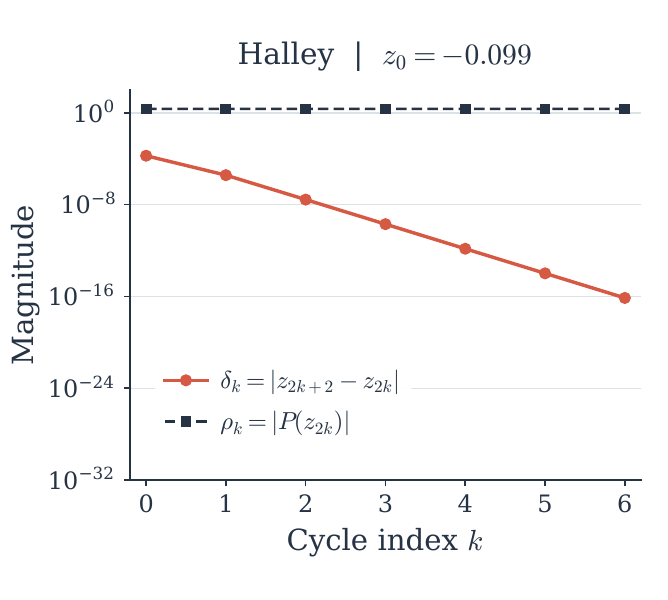}
\caption{Two-step difference $\delta_k$ and root residual $\rho_k$ along representative real orbits, on common logarithmic axes. Rapid approach to a two-cycle coexists with a nonvanishing root residual.}
\label{fig:orbits}
\end{figure}

\section{Discussion}\label{sec:discussion}

We have exhibited an explicit quartic polynomial with distinct roots for which
both Newton's and Halley's methods possess extraneous attracting two-cycles. This complements the construction of Collins and
Hunter~\cite{CollinsHunter}, who produced, for every $d\ge5$, polynomials for
which Newton's, Halley's, and Schr\"oder's methods share the same
superattracting two-cycle $\{0,1\}$. Thus, for the Newton--Halley pair, our
quartic example fills the degree-four gap left by their construction. There is,
however, an essential difference: in our example the attracting cycles of the
two methods are distinct. The following result shows that this is
unavoidable in degree four.

\begin{proposition}[No common attracting two-cycle in degree four]
\label{prop:no-common-two-cycle}
A quartic polynomial cannot have a common extraneous
two-cycle that is attracting for both Newton's and Halley's methods.
\end{proposition}

The proof is given in Appendix~\ref{app:ausiliary}. More precisely,
the calculation there shows that any common extraneous two-cycle
of the Newton and Halley maps of a quartic has Halley multiplier
$36$, and is therefore repelling for Halley's method. Thus the
use of distinct cycles in our example is not merely a consequence
of the parameter selected by the search.

Distinct cycles do not require disjoint failure basins. For the
polynomial $P$ of Theorem~\ref{thm:quartic}, it can also be checked
directly that $z_0=-1/2$ is attracted to the respective extraneous
two-cycle under each method. This can be verified by following the
two orbits into trapping neighborhoods of the corresponding cycles;
we omit the details. Since both basins are open, their intersection
contains an open neighborhood of this initial value. Running the
two methods separately from the same starting point therefore need
not avoid extraneous attraction.

The search procedure can also be extended to include additional
root-finding methods. Within the family~\eqref{eq:family}, the
superattracting Newton cycle is already prescribed, so one can
screen the remaining iterations for extraneous attracting cycles,
allowing their periods to differ. For example, consider the
Schr\"oder map used by Collins and Hunter~\cite{CollinsHunter},
also known as Chebyshev's method,
\[
S_p(z)=z-\frac{p(z)}{p'(z)}
       -\frac{p(z)^2p''(z)}{2p'(z)^3}.
\]
Increasing the denominator bound in Algorithm~\ref{alg:search} to
$500$ and adding a periodic-orbit test for $S_p$, including period
three, leads to the parameter $a=-576/395$ and the polynomial
\[
Q(z)=395p_{-576/395}(z)
    =395z^4-576z^3-33z+33.
\]
Newton's map retains the exact superattracting cycle
$0\leftrightarrow1$. Numerical computations identify an extraneous
attracting two-cycle for Halley's method and an extraneous attracting three-cycle for Schr\"oder's method.
The estimated cycle multipliers are
$
\mu_H\approx0.091,
$ and $
\mu_S\approx-0.005$.
These computations indicate simultaneous failure of all three
methods on the same quartic, through cycles of periods $2$, $2$,
and $3$.

The two numerically located cycles for $Q$ may be certified by the
same rational interval argument used in
Section~\ref{sec:analytic-proof}. The additional interval bounds and certificates
are not included here.

For the Newton--Halley pair, the degree-three question remains.
Collins and Hunter~\cite{CollinsHunter} conjectured that, for every
cubic polynomial, at least one of Newton's and Halley's methods is
generally convergent. A proof of that conjecture, together with
Theorem~\ref{thm:quartic}, would establish degree four as the least
degree in which simultaneous failure of almost-everywhere root
convergence can occur for the pair.

\appendix
\section{Auxiliary proofs}\label{app:ausiliary}
\begin{proof}[Proof of Lemma \ref{lem:contraction}]
The map $h^{\circ2}$ sends $I$ into itself and satisfies
\[
\sup_I |(h^{\circ2})'|
\le
\left(\sup_I |h'|\right)\left(\sup_J |h'|\right)<1.
\]
Hence $h^{\circ2}$ has a unique fixed point $t_1\in I$. Setting
$t_2=h(t_1)\in J$, we obtain a two-cycle because $I\cap J=\varnothing$.
Its multiplier has modulus less than one, so the cycle is attracting.
\end{proof}

\begin{proof}[Proof of Lemma \ref{lem:bounds}]
We obtain the polynomial bounds by monotonicity and evaluation at
the interval endpoints. The polynomial $F$ is positive and strictly
increasing on $(0,\infty)$. Also,
\[
K'(t)=23040t^3+25200t^2+7350t-96>0
\]
on $I\cup J$, since $t\ge49/500$ there and
$7350(49/500)>96$.

For $D$, differentiation gives $D'=12S$, where
\[
S(t)=5760t^5+10500t^4+4900t^3+48t^2-26t-35.
\]
On $I$, the bound $t\le1/10$ yields
\[
S(t)
\le
\frac{5760}{10^5}+\frac{10500}{10^4}
+\frac{4900}{10^3}+\frac{48}{10^2}-35
<0.
\]
On $J$, retaining only the cubic term among the positive terms gives
\[
S(t)
\ge
4900\left(\frac{259}{500}\right)^3
-26\left(\frac{519}{1000}\right)-35
>0.
\]
Consequently, $D$ decreases on $I$ and increases on $J$.

Evaluation at the rational endpoints now gives all the bounds in
Table~\ref{tab:bounds}. In particular, $D$ does not vanish on either
interval, so $h$ has no poles there. Substituting these bounds
into~\eqref{eq:hprime}, we obtain
\[
\sup_I|h'|
<
\frac{24(9/4)^2(6/5)}{32^2}
<\frac15,
\qquad
\sup_J|h'|
<
\frac{24(39/4)^2\,2500}{1990^2}
<\frac32.
\]
\end{proof}
\begin{proof}[Proof of Proposition \ref{prop:no-common-two-cycle}]
Suppose, by contradiction, that a quartic polynomial $p$ has a common
extraneous attracting two-cycle for Newton's and Halley's maps. By an
affine change of variable we may assume that this cycle is
\[
0\longleftrightarrow 1.
\]
Indeed, if $A(z)=\alpha+(\beta-\alpha)z$ sends $0$ and $1$ to the two
points $\alpha,\beta$ of the cycle, then for $q=p\circ A$ one has
\[
N_q=A^{-1}\circ N_p\circ A,
\qquad
H_q=A^{-1}\circ H_p\circ A.
\]
Thus the multiplier of the cycle is unchanged.

Since the cycle is extraneous,
\[
p(0)p(1)\neq0.
\]
Moreover, because both points are mapped to finite points by $N_p$, we
also have
\[
p'(0)p'(1)\neq0.
\]
At each $x\in\{0,1\}$, the two methods have the same image, hence
\[
N_p(x)=H_p(x).
\]
Using
\[
N_p(x)=x-\frac{p(x)}{p'(x)},
\qquad
H_p(x)=x-
\frac{2p(x)p'(x)}
{2p'(x)^2-p(x)p''(x)},
\]
and cancelling $p(x)\neq0$, we obtain
\[
\frac1{p'(x)}
=
\frac{2p'(x)}
{2p'(x)^2-p(x)p''(x)}.
\]
It follows that
\[
p(x)p''(x)=0,
\]
and therefore
\[
p''(0)=p''(1)=0.
\]

Since $p$ has degree four, $p''$ is quadratic. Hence, for some
$A\neq0$,
\[
p''(z)=A z(z-1).
\]
Integrating twice,
\[
p'(z)
=
A\left(\frac{z^3}{3}-\frac{z^2}{2}\right)+B,
\]
and
\[
p(z)
=
A\left(\frac{z^4}{12}-\frac{z^3}{6}\right)+Bz+C.
\]
The relations $N_p(0)=1$ and $N_p(1)=0$ are respectively equivalent to
\[
p(0)=-p'(0),
\qquad
p(1)=p'(1).
\]
The first gives $C=-B$, while the second then gives
\[
B=\frac{A}{12},
\qquad
C=-\frac{A}{12}.
\]
Thus, up to multiplication by a nonzero constant,
\[
p(z)=z^4-2z^3+z-1.
\]

It remains to inspect the common cycle for Halley's map. For any
polynomial,
\[
H_p'(z)
=
\frac{p(z)^2
\bigl(3p''(z)^2-2p'(z)p'''(z)\bigr)}
{\bigl(2p'(z)^2-p(z)p''(z)\bigr)^2}.
\]
At $z=0,1$ we have $p''(z)=0$, so
\[
H_p'(z)
=
-\frac{p(z)^2p'''(z)}{2p'(z)^3}.
\]
For
\[
p(z)=z^4-2z^3+z-1
\]
one has
\[
p'''(z)=24z-12,
\]
and a direct evaluation yields
\[
H_p'(0)=H_p'(1)=6.
\]
Consequently the multiplier of the common two-cycle for Halley's map is
\[
\mu_H=H_p'(0)H_p'(1)=36.
\]
The cycle is therefore repelling for Halley's method, contradicting the
assumption that it is attracting for both methods.
\end{proof}

\end{document}